\documentclass{amsart}
\usepackage[utf8]{inputenc}
\usepackage{amsfonts,amssymb,amsmath,amsthm}
\usepackage{hyperref}
\usepackage[Conny]{fncychap}
\usepackage{spreadtab} % cette commande gere l'addition verticale avec 3 termes
\usepackage{xcolor}

\usepackage[T1, OT1]{fontenc}
\theoremstyle{plain}
\newtheorem{theorem}{\textbf{Theorem}}[section]

\newtheorem{proposition}{\textbf{Proposition}}[section]
\newtheorem{lemma}{\textbf{Lemma}}[section]
\newtheorem{corollary}{\textbf{Corollary}}[section]

\theoremstyle{definition}
\newtheorem{definition}{\textbf{Definition}}[section]
\theoremstyle{remark}

\newtheorem{example}{\textbf{Example}}[section]

\numberwithin{equation}{section}

\def\ZZ{\mathbb{Z}\setminus\{0\}}
\def\NN{\mathbb{N}\setminus\{0\}}

\def\(Z/nZ)(Z/nZ){\mathbb{(Z/nZ)}\setminus\{0, 1\}}

\def\deg{deg}

\def\Q{\mathbb{Q}}
\def\C{\mathbb{C}}
\def\F{\mathbb{F}}

\def\Z{\mathbb{Z}}
\def\R{\mathbb{R}}

\def\G{\mathcal{G}}

\def\bX{\mathbf{X}}

\def\bb{\mathbf{b}}

\def\bx{\mathbf{x}}
\def\bd{\mathbf{d}}

\def\bu{\mathbf{u}}
\def\bv{\mathbf{v}}
\def\bw{\mathbf{w}}

\DeclareMathOperator{\ord}{ord}
\DeclareMathOperator{\Nr}{N}
\DeclareMathOperator{\Tr}{Tr}

\begin{document}
% \tableofcontents
%\cite{Lamport}
%\bibliographystyle{plain}
%\bibliography{mabiblio,bibliofac, commun}

\title[]{The $p$-powers dividing certain exponential sums}

\author{Antonine Phigareau}

\address{D{\'e}partment  de Math{\'e}matiques et Statistique,   Universit{\'e} de
Montr{\'e}al, CP 6128 succ Centre-Ville, Montr{\'e}al, QC  H3C 3J7, Canada.}

\email{phigareauantonine@gmail.com} 
%\date{}

\begin{abstract} 

We define the notion of couple density $(D, \bb)$ where $D$ is a non-empty subset of $\Z^{m}$ and $\bb$ a fixed element in $\{0, \cdots, q-2\}^{m};$
We determine a minimum in terms of the density of the couple $(D, \bb)$ for the $q$-adic valuation of the sum $ S_{\ell}(F,\bb)$ with $F$ a Laurent polynomial. And we show that this minimum is a bound for the $q$-adic valuation of the zeros and poles of the associated $L$-function.
\end{abstract}

\maketitle

\textbf{Keywords}

\noindent Finite fields, exponential sums, mixed characters, $L$-functions, divisibility, $p$-adic valuation, sequences of sums.

\hfill

\noindent \textbf{Classification}

$ 11 \text{T}23, ~ 11 \text{T}24, ~ 11 \text{M}38  $

\hfill

\noindent\textbf{Thanks}

\noindent I thank Mr. Régis Blache and Professor Andrew Granville for their support in the preliminary to this article.

\section{Introduction}

\noindent Given a prime number $p$ and  an integer $f\geq 1$, let $ \F_{q}$ be the finite field with $q=p^{f}$ elements, and let $\F_{q^{\ell}}$ be the field extension of degree $\ell$ inside  a fixed algebraic closure, for each  $ \ell \geq 1$. Let $\psi$ denote a non-trivial additive character  of $ \F_{q}$ and let $\omega$ be the Teichm\"uller character of $ \F_{q}$ in $\Q_{p}(\zeta_{q-1})$. 
Let 
\[  \pi=\psi(1)-1; \]
it is known  that $p$ and $\pi^{p-1}$ are associates in $\Z_{p}$.

Given $D:=\{\bd_{i}\}_{1\leq i \leq n}\subset \Z^{m}$ with $\bd_{i}=(d_{i1},\ldots, d_{ im})$, let 
\[
 F(\bx):=\sum^{n}_{i=1}{a}_{i}\bx^{\bd_{i}} \in  \F_{q}[x_{1}, x^{-1}_{1},\ldots,x_{m}, x^{-1}_{m}] 
 \]
 be a Laurent polynomial in $ m $ variables,  where  we define 
 \[
  \bx^{\bd}=x_1^{d_{1}}\ldots x_m^{d_{m}} \text{ for } \bx=(x_{1},\ldots, x_{m}) \text{ and }\bd=(d_{1},\ldots, d_{m}).
  \]
For each  $\bb=(b_{1}, \ldots, b_{m})\in \{0, \ldots, q-2\}^{m}$  such that $\bx^{\bb}=x_1^{b_{1}}\ldots x_m^{b_{m}},$  we define the character sum
\[
 S_{\ell}(F,\bb)= \sum_{\bx \in (\F^{\times}_{q^{\ell}})^{m}}\omega \big(\Nr_{\F_{q^{\ell}}/ \F_{q}}(\bx^{\bb})\big)\psi\big(\Tr_{\F_{q^{\ell}}/ \F_{q}}(F(\bx))\big)
 \]
  and the $L$-function 
\[
L(F,\bb, T)=\exp \Bigg (\sum_{\ell\geq 1}S_{\ell}(F,\bb)\frac{T^{\ell}}{\ell }\Bigg)
\]
which  is known to be a rational fraction in $T$ \cite{B66} of the form
\[
 L(F,\bb, T) =\frac{\prod^{s}_{i=1}(1-\beta_{i}(F, \bb)T)}{\prod^{t }_{j=1}(1-\alpha_{j}(F, \bb)T)}
 \]
for $\alpha_{j}(F,\bb), \beta_{i}(F, \bb)\in \C$.
We will determine a lower bound for the $q$-adic valuations of the poles and roots of this rational $L$-function. 

\subsection*{Results in context} 

In 1935 C. Chevalley proved Artin's conjecture that if  $F(\bx)$ is   homogeneous in $m$ variables  of total degree $ d$ where $d<m $ then $F$ has a non-trivial zero.
Chevalley subsequently relaxed  the homogeneity assumption and then E. Warning showed that $p$, the   characteristic of the field, divides the number of zeros of $F(\bx)$ specifically:

\begin{theorem}[Chevalley-Warning Theorem]
Suppose that $F_{1},\cdots,F_{r}\in \F_{q}[X_{1},\cdots,X_{m}]$ where $d_i=\deg(F_{i})$.
If  $\sum^{r}_{i=1} d_i< m$ then $p$ divides $\#V$ where
$$ V:= \{ \bx \in \F^{m}_{q}: F_1(\bx)=\cdots = F_r(\bx)=0 \}.$$.
\end{theorem} 
The proof of the Chevalley-Warning theorem can be found in \cite{A64},\cite{J73} and \cite{RH12}.

In 1964  Ax \cite{A64}, improving Dwork  \cite{D62}, proved that $q^{\lambda}$ divides $\#V$
where
$$ \lambda= \bigg \lceil\frac{m-\sum^{r}_{i=1}d_{i}}{\sum^{r}_{i=1}d_{i}}\bigg \rceil$$
(and   $\lceil a \rceil $ is the smallest integer greater than or equal to $a$),
which was improved by Katz \cite{K71}  in 1971 to divisibility by $q^\mu$ where 
$$ \mu= \bigg\lceil\frac{m-\sum^{r}_{i=1}d_{i}}{\max_{1\leq i\leq r} \{ d_{i} \}}\bigg\rceil .$$
 
In $1987$, some interesting techniques  intoduced by Adolphson and Sperber \cite{AS87} in connection with sums of additive characters where the multiplicative character is trivial. For example he considered the sums 

\[
  S(F):=\sum_{(x_{1},\ldots, x_{m})\in\F_{q}^{m}}\psi(F(x_{1},\ldots, x_{m})).
 \]
%in which there is no multiplicative character (like the Teichmuller character above).
The Newton polyhedron, $\Delta(F)$, is the convex hull of the set 
\[
 \{ \bd_{i}: 1\leq i\leq n\}\cup \{(0,\ldots, 0)\}\subset \R^{m}.
 \]
  Let $ \omega(F)$ be the smallest positive rational number such that
$ \omega(F)\cdot \Delta(F)$ contains a point of $(\NN)^{m}$.
Adolphson and Sperber proved that if $F$ is not  a polynomial   in a subset of the variables $ X_{1},\ldots,X_{m}$ then 
$$ v_{p}(S(F))\geq f \omega(F).$$
When this bound is used in the case of the number of points, we obtain the  Ax-Katz result.

In 1993, Adolphson, Sperber \textit{ et al.} \cite{AS93}   extended the results of \cite{AS89, AS90} to exponential sums involving multiplicative as well as additive characters for $\ell=1$ when the polynomial $F$ is a Laurent polynomial. See also D.Wan \cite{W04} and Adolphson and Sperber \cite{AS14} for other developments.

It should be noted that additive sums $s(F)$ for prime fields ($q=p$) were previously studied by Chen and Cao \cite{JW10}. By an elementary and logarithmic approach, using the degree matrix, they improved the Adolphson-Sperber theorem, and, consequently, the Ax-Kartz theorem.

\noindent In this article, we understand the $q$-adic valuation of the mixed sum of characters $S_{\ell}(F,\bb_{\ell})$ using the associated $L$-functions. Theorem \ref{vp} is a fundamental result that clearly specifies the nature of this estimate by assuming maximal support and minimal weiht, primarily through Proposition \ref{prop}. Moreover, we study how the $p$-adic valuation of $S_{\ell}(F,\bb_{\ell})$ varies for very large $\ell$. By studying a congruence for $S_{\ell}(F,\bb_{\ell})$ modulo powers of $\pi$ (\ref{S}) given by Adolphson and Sperber [\cite{AS14}, 6.3]  we get a clearer idea of the power of $p$ that can divide 
$S_{\ell}(F,\bb_{\ell})$, which  avoids having to sort the solutions of the system of modular equations based on their zero coordinates. This approach for mixed character sums works as well as replacing the ordinary polynomial with a Laurent polynomial. This develops 2012 work of  Blache \cite{B12} who woked  with polynomials (rather than Laurent polynomials) and with $\bb=0$.
 
We will generalize Blache's work \cite{B12} to the case $\bb\neq 0$ in two ways: First we work in 
 any extension $ \F_{q^{\ell}}$, rather than  $\F_{q}$ in Blache's work. Secondly Blache  considers only the simple character sums of the form $S(F)$ rather than the mixed character sums of the form $S(F,\bb)$.  The methods used involve a marriage of number theory and combinatorial optimization.

\noindent Given an integer $t$ that is written in base $p$ in the form
$$ t= \sum^{f}_{i=0}p^{i}a_{i}, \text{ for } a_{i}\in\{0, \ldots,p-1\},$$
we define the $p$-weight of $t$ by
\[ \sigma_{p}(t):= \sum^{f}_{i=0}a_{i},\] 
as well as the quantity
\[\rho_{p}(t):= \prod^{f}_{i=0}a_{i}!  .\]

For each $ \ell\geq 1 $, let $\bb_{\ell}=\frac{q^{\ell}-1}{q-1}\bb,$   and   consider the set
\begin{equation}
L(D, q^{\ell}, \bb_{\ell})= \big \{ \bu=(u_{1},\ldots, u_{n})\in \{0,1,\ldots,q^{\ell}-1\}^{n};\sum^{n}_{i=1} u_{i}\bd_{i}+\bb_{\ell}\equiv 0\mod(q^{\ell}-1)\big \}.
\end{equation}
If $\bu=(u_{1},\ldots, u_{n})\in L(D, q^{\ell}, \bb_{\ell})$  we define the $p$-weight of $\bu$ by
$$
\sigma_{p}(\bu):= \sum^{n}_{i=1}\sigma_{p}(u_{i})$$ 
and, if $L(D, q^{\ell}, \bb_{\ell})$ is non-empty then
 $$\sigma_{p}(D, q^{\ell}, \bb_{\ell}):=\min \{\sigma_{p}(\bu), \bu\in L(D, q^{\ell}, \bb_{\ell})\}.
$$
For $\bb\neq 0$ so that $\sigma_{p}(\bu)\neq 0$, we will show that   
$$ v_p(S_{\ell}(F, \bb))\geq  \frac{\sigma_{p}(D,q^{\ell},\bb_{\ell})}{p-1} $$

For this, it will be necessary to extend the notion of $p$-density introduced in \cite{B12} for the additive sum to this situation; we will show the following theorem
\begin{theorem} [Key result] \label{vp}
The sequence
$$ \Bigg (\frac{\sigma (D,q^{\ell},\bb_{\ell})}{f\ell(p-1)} \Bigg)_{\ell\geq 1}$$
admits a minimum. And that $$ v_p(S_{\ell}(F, \bb))\geq  \frac{\sigma_{p}(D,q^{\ell},\bb_{\ell})}{\ell f(p-1)}  \text{ for all } \ell \geq 1.$$
\end{theorem}
The minimum can be determined by the method given in proposition \ref{p30}. \\

For $\bu\in L(D, q^{\ell}, \bb_{\ell})$, the $p$-density of $\bu$ is the rational number defined by
$$s_{p}(\bu):=\frac{\sigma_{p}(u)}{f\ell(p-1)}$$ 
and the $p$-density of the couple $(D,\bb)$ is the rational number defined by
$$s_{p}(D,\bb):= \min \bigg \{\frac{\sigma_{p}(D,q^{\ell},\bb_{\ell})}{f\ell(p-1)} , \ell \geq 1 \bigg\}.$$
The lower bound for the sums of characters $S_{\ell}(F,\bb)$ will be given by the following Proposition and Theorem:

\begin{proposition} \label{p1}
We have
\[
 v_{q}( S_{\ell}(F,\bb)) \geq \ell s_{p}(D,\bb) \text{ for all } \ell\geq 1.
 \]
\end{proposition}
The proof of this Proposition follows from Proposition \ref{prop}.

Now, since $ v_{\pi}(x) = f(p-1)v_{q}(x)$ for all $ x \in \Q_{p}(\zeta_{q-1})$, we can state the following theorem.
\begin{theorem} 
Let $ H $ be the vector subspace of 
$ \F_{q}[X_{1}, X^{-1}_{1},\ldots,X_{m}, X^{-1}_{m}]$
 generated by monomials
$ \bX^{\bd_{1}},\ldots,\bX^{\bd_{n}}$ with each $\bd_{i}\in \Z^m$.
Let $ \bb\in \{0,\ldots, q-2\}^{m}$ such that $ \bb\neq 0$.
There exists a  polynomial $ F $ in $ H $ such that $ v_\pi(S_{1}(F,\bb))=\sigma(D, q,\bb)$.
\end{theorem}
The proof of this theorem is similar to that of Theorem \ref{t28}.

We again assume that all the variables appear in the Laurent polynomial $F$. 
We define $\mu(F, \bb)$ to be the minimum of the 
$v_{q}(\alpha_{i}(F, \bb))$ and $ v_{q}(\beta_{j}(F,\bb)) $, varying over the zeros and poles of 
\[
L(F,\bb, T)=\exp \Bigg (\sum_{\ell\geq 1}S_{\ell}(F,\bb)\frac{T^{\ell}}{\ell }\Bigg)  =\frac{\prod^{s}_{i=1}(1-\beta_{i}(F, \bb)T)}{\prod^{t }_{j=1}(1-\alpha_{j}(F, \bb)T)}.
 \]

We will show that this minimum is a lower bound for the $q$-adic valuation of the zeros and poles of the $L$-function through the following theorem:

\begin{theorem} 
We have $\mu(F, \bb) \geq s_{p}(D, \bb)$.
\end{theorem}
For the proof, see Theorem \ref{pL1}.

%%%%%%%%%%%%%%%%%%%%
\section{A bound for the valuation of sums of characters}

For   any integer $ \ell\geq 1 $, we  consider the set
\begin{equation} \label{L}
L(D, q^{\ell}, \bb_{\ell})= \big \{\bu \in \{0,1,\ldots ,q^{\ell}-1\}^{n}: \sum^{n}_{i=1} u_{i}\bd_{i}+\bb_{\ell}\equiv 0\mod(q^{\ell}-1)\big \}.
\end{equation}
where $\bu=(u_{1},\ldots, u_{n})$ and $ \bb_{\ell}=\frac{q^{\ell}-1}{q-1}\bb$. This congruence condition on vectors can be written as the system of simultaneous congruences
\begin{equation} \label{eq; congruence}
\sum^{n}_{i=1} u_{i}d_{i,j}+\frac{q^{\ell}-1}{q-1} b_j \equiv 0\mod(q^{\ell}-1)  \text{ for } j=1,\dots,m.
\end{equation}

Recall that 
\[
 F(\bx):=\sum^{n}_{i=1}{a}_{i}x_{1}^{d_{i,1}}\ldots x_{m}^{d_{i,m}}
\]
We will assume that for each $j$ there exists $i$ for which $d_{i,j}\ne 0$ else the variable $x_j$ does not appear in the polynomial $F(\bx)$.  It is worth noting that if, on the other hand, $d_{ij}=0$ for all $i$ but $b_j=1$ then $ L(D, q^{\ell}, \bb_{\ell})=\emptyset$ for all $\ell \geq 1$, and so   $ S_{\ell}(F,\bb)=0$.

 %%%%%%%%%%%%%%%%%%%
\subsection{ When is $ L(D, q^{\ell}, \bb_{\ell})=\emptyset$?}
 
Throughout this section we assume that the  $d_{i,j}$ and the $b_j$ are given integers.

\begin{lemma} \label{LD1}
Suppose that $A:=(d_{i,j})_{1\leq i\leq n, 1\leq j\leq m}$. If $m=n$ and $(\det A, q-1)=1$ then there exist infinitely many integers $\ell\geq 1$ such that  $ L(D, q^{\ell}, \bb_{\ell})$ is not empty.
\end{lemma}

\begin{proof} 
 We wish to find $\bu$ for which $A\bu \equiv -\frac{q^{\ell}-1}{q-1}\bb \mod(q^{\ell}-1)$. Multiplying through by $A^{-1}$ we obtain $\det A\cdot \bu \equiv -\frac{q^{\ell}-1}{q-1}A^{-1}\bb \mod(q^{\ell}-1)$.
 If $(\det A,q^{\ell}-1)=1$ then we can invert $\det A \mod(q^{\ell}-1)$ and solve to find $\bu$.
 
 For each prime $p$ dividing $\det A$ then let $\ell_p$ denote the order of $q \pmod p$, so that 
 $p$ divides $q^\ell-1$ if and only if $\ell_p$ divides $\ell$. Note that each $\ell_p>1$ as $(\det A, q-1)=1$.
  Let $r=\prod_{p\mid \det A} \ell_p$ so if $(r,\ell)=1$ then $\ell_p$ does not divide $\ell$ for all $p\mid \det A$ and so $(\det A,q^{\ell}-1)=1$. Therefore, for example, we can take all integers $\ell\equiv 1 \mod r$.
\end{proof}

Define  $\Lambda(D,\bb):=\{ \ell\geq 1:  L(D, q^{\ell}, \bb_{\ell}) \ne \emptyset\}$.
 \begin{lemma}
 If $\ell \in \Lambda(D,\bb)$ then $k\ell \in \Lambda(D,\bb)$ for all integers $k\geq 1$.
 \end{lemma}
 \begin{proof}
 If $ \sum^{n}_{i=1} u_{i}\bd_{i}+\bb_{\ell}\equiv 0\mod(q^{\ell}-1)$ then let $U_i=u_i \frac{q^{k\ell}-1}{q^\ell-1}$, and multiplying through by $\frac{q^{k\ell}-1}{q^\ell-1}$ we establish $\bu \cdot \frac{q^{k\ell}-1}{q^\ell-1} \in L(D, q^{k\ell}, \bb_{k\ell}) $ as $\bb_{\ell}\cdot  \frac{q^{k\ell}-1}{q^\ell-1} =\bb_{k\ell}$.
 \end{proof}
 
 \begin{lemma} \label{lem: CRT}
 $\bu\in  L(D, q^{\ell}, \bb_{\ell}) $ if and only if for every prime power $r^e\| q-1$ (with $r$ prime, $e\geq 1$) there is a solution to $\sum^{n}_{i=1} v_{i}\bd_{i}\equiv r^f \bb\mod r^{e+f}$
 where $r^f\| \frac{q^{\ell}-1}{q-1}$.
 \end{lemma}
 
 \begin{proof} There are solutions to \eqref{eq; congruence} if and only if there are solutions to
 $ \sum^{n}_{i=1} u_{i}\bd_{i}+\bb  \frac{q^{\ell}-1}{q-1} \equiv 0\mod r^{e+f} $ for every $r^{e+f}\| q^\ell-1$
 (including the cases where $e=0$) by the Chinese Remainder theorem.
 
 Now if $e=0$ then we can take each $u_i\equiv 0 \mod r^{f} $, so we may assume that $e\geq 1$.
 Write $\frac{q^{\ell}-1}{q-1}=Ar^f$ where $r\nmid A$. Then the solutions to
 $ \sum^{n}_{i=1} u_{i}\bd_{i}\equiv -Ar^f \bb  \mod r^{e+f} $ are in 1-to-1 correspondence with the solution to 
  $\sum^{n}_{i=1} v_{i}\bd_{i}\equiv r^f \bb\mod r^{e+f}$ by taking each $u_i= -Av_i $ and 
  each $v_i\-Bu_i \mod r^e$  where $AB\equiv -1 \mod r^{e+f} $.
   \end{proof}
 
  \begin{corollary} If $\Lambda(D,\bb)$ is non-empty then
  there exists an integer $\ell\geq 1$ such that $\Lambda(D,\bb)=\ell \mathbb Z_{\geq 1}$, where all of the prime factors of $\ell$ divide $q-1$.
  \end{corollary}
  
\begin{proof} For each prime $r$ dividing $q-1$ select $r^f$ minimally for which there is a solution to  
 $\sum^{n}_{i=1} v_{i}\bd_{i}\equiv r^f \bb\mod r^{e+f}$.  There are solutions to $\sum^{n}_{i=1} v_{i}\bd_{i}\equiv r^F \bb\mod r^{e+F}$ if and only if $f\leq F$.  Let $m=\prod_{r|q-1}r^{e+f}$.
  By lemma \ref{lem: CRT} we deduce that 
 $h\in \Lambda(D,\bb)$ if and only if $m$ divides $q^h-1$ which holds if and only if $h$ is divisible by 
 $\ord_m(q):=\ell$.
\end{proof}

\begin{proposition} To determine whether  $\Lambda(D,\bb)$ is  non-empty we proceed as follows:
Perform Gaussian elimination without  divisions (as described in the proof) on the system of equations 
\[
\sum_{i=1}^{n} v_{i}\bd_{i}=  \bb x \text{ over } \Z[x]; 
\]
to obtain an ``effectively diagonal'' system of equations\footnote{This means that there is a set of integers
$i_1=1<i_2<\dots<i_{M}$ where $d'_{i,j} = 0$ if $i<i_j$ and $d'_{i_j,j} \ne 0$.}
\[
\sum_{i=1}^{n} v_{i}\bd'_{i}=  \bb' x 
\]
 together with some equations not involving the variables $v_i$, simply of the from 
 \[
 0=\beta_jx    \text{ for } M<j\leq M'.
 \]
Then $\Lambda(D,\bb)$ is  non-empty if and only if $q-1$ divides $\beta_j$ for each $j, M<j\leq M'$
 \end{proposition}

\begin{proof}
By Lemma \ref{lem: CRT} we know that $\Lambda(D,\bb)$ is  non-empty if and only if 
for any prime power $r^e\| q-1$ there exists an integer $f\geq 0$ for which there are solutions to 
\[ \sum^{n}_{i=1} v_{i}\bd_{i}\equiv r^f \bb\mod r^{e+f}. \]
We work with the system of equations
\[
\sum^{n}_{i=1} v_{i}d_{i,j}= b_jx    \text{ for } j=1,\dots,m
\]
and will eventually substitute in $x=r^f $.
We will perform Gaussian elimination without using divisions: Let  
\[
\Delta_1=\gcd(d_{1,1},d_{1,2},\dots,d_{1,m})
\]
 so there exist integers
$c_1,\dots,c_m$ for which $\Delta_1=c_1d_{1,1}+\dots+c_md_{1,m}$, and so we have
\[
\Delta_1v_1+\sum_{i=2}^{n} v_{i}(c_1d_{i,1}+\dots+c_md_{i,m}) = (c_1b_1+\dots+c_mb_m)x.
\]
Now $\Delta_1$ divides all the $d_{1,j}$ so we can subtract a multiple of our new equation from each of the others and remove the $v_1$ variable.  We now do this for each variable in turn, so we end up with an ``effectively diagonal'' system of equations which looks like
\[
\sum_{i=i_j}^{n} v_{i}d_{i,j}'= b_j'x    \text{ for }1\leq j\leq M
\]
 where $i_1=1<i_2<\dots<i_{M}$ with each $d_{i_j,j} \ne 0$, together with some equations not involving the variables $v_i$, simply of the from 
 \[
 0=\beta_jx    \text{ for } M<j\leq M'.
 \]
  Thus our original system of congruences requires that 
 $\beta_j r^f\equiv 0 \mod r^{e+f}$ and so there are   solutions if and only if $r^e| \beta_j$ for each $r$ dividing $q-1$ and this holds if and only if    $q-1| \beta_j$ for each $j, M<j\leq M'$. 
 
 If $q-1$  divides every $\beta_j$ then we can discard the equations $\beta_j r^f\equiv 0 \mod r^{e+f}$ as satisfied and we claim that there are always solutions to the effectively diagonal system of equations:
  To begin with we set $v_i=0$ if $i\ne i_j$ for all $j$, that is if $v_i$ is an independent variable. Then, after re-numbering, we are left with a square upper diagonal matrix, so a system of equations of the form 
  \[
  \sum_{i=j}^{m} v_{i}\Delta_{i,j}= B_jx \text{ for } 1\leq j\leq M
  \]
   where
 $\Delta_{i,i}\ne 0$ for all $i$ but $\Delta_{i,j}=0$ if $i<j$, that is $\Delta v=B x$ as matrices. Now the matrix $\Delta$ is invertible over $\Q$ so there exists an $M$-by-$M$ matrix $C$ with integer coefficients such that $C\Delta=NI$ for some integer $N$. Therefore 
 $N v = C\Delta v=CB x$; that is, $N v_i = \sum_j C_{i,j}B_j x$.  
 Now if $N=r^f s$ where $r\nmid s$ then we can let $x=r^f$ and $st\equiv 1 \mod p^e$ so that we can let 
 $    v_i =t \sum_j C_{i,j}B_j   $ for each $i$. 
 
  This shows that there are solutions for some $f$ as desired.  If we let $Q=\prod_{r|q-1} r^{v_r(N)}$ then we obtain a solution for any $\ell$ for which $Q$ divides $\frac{q^\ell-1}{q-1}$
 \end{proof}

 For $\bu \in L(D, q^{\ell}, \bb_{\ell})$ we define  
\[
 \sigma_{p}(\bu )=\sum^{n}_{i=1}\sigma_{p}(u_{i}) \text{ and }  \sigma(D,q^{\ell},\bb_{\ell}):=\min \{\sigma_{p}(\bu),\bu \in L(D, q^{\ell }, \bb_{\ell})\}. \label{pm}
\]
If $ \bb = 0 $, we have  $0  \in L(D, q^{\ell }, \bb_{\ell})$ so $ \sigma (d, q^{\ell}, \bb_{\ell}) = 0 $.
Therefore henceforth we will assume that  $ \bb \neq 0$.

With the above notation, if $\bx \in (\F^{\times}_{q^{\ell}})^{m}$ then $ \omega ( \Nr_{\F_{q^{\ell}}/ \F_{q}}(\bx^{\bb}))=\bx^{\bb_{\ell}}$ end $\omega$  the Teichm\"uller character of $ \F_{q}$ in $\Q_{p}(\zeta_{q-1})$ .
 Adolphson and Sperber [\cite{AS14}, $6.3$] used this to show that
\begin{equation}
\label{S}
 S_{\ell}(F,\bb_{\ell})\equiv \sum_{\bu \in G} \prod^{n}_{i=1}\Bigg(\frac{\omega(a_{i })^{u_{i}}}{\rho_{q}(u_{i})}\Bigg)\pi^{ \sigma (D,q^{\ell},\bb_{\ell})} \mod \pi^{ \sigma(D,q^{\ell},\bb_{\ell})+1}
\end{equation}
where $G= \{ \bu  \in L(D, q^{\ell}, \bb_{\ell}):\   \sigma_{p}(\bu)=\sigma(D, q^{\ell},\bb_{\ell} )\}$.
Result (\ref{S}) differs from those of Adolhson and Sperber, particulary in notation and simplification. For example, their polynomial is denoted $ f_{\lambda}=\sum^{n}_{j=1}\lambda_{j}\bx^{a_{j}},$ with $ \lambda = (\lambda_{1},\ldots,\lambda_{n})\in \F_{q}^{n} $ and the $p$-weight is indicated as $\omega_{p}.$\\

From congruence (\ref{S}) and the fact that $v_{\pi}(x)=(p-1)v_p(x),$ for all $ x \in \Q_{p}(\zeta_{q-1})$, we obtain:
\[
v_p(S_{\ell}(F,\bb)) \geq \frac{\sigma (D, q^{\ell},\bb_{\ell}) }{p-1} .
\] 

 We will extend the notion of $p$-density, introduced in \cite{B12} for additive sums, to this situation by showing that the sequence 
\[
 \Bigg (\frac{\sigma (D,q^{\ell},\bb_{\ell})}{f\ell(p-1)} \Bigg)_{\ell\geq 1}
 \]
admits a minimum and that this minimum is a lower bound for the $q$-adic valuation of the zeros and poles of the $L-$function. 

We need to study the solutions to congruence systems when $\ell $ varies:
  For each $\ell \geq 1 $ define
$\delta_{\ell}:\{0,\ldots, q^{\ell}-1\}\to \{0,\ldots, q^{\ell}-1\} $ 
so that $\delta_{\ell}(q^{\ell}-1)=q^{\ell}-1$ and otherwise $\delta_{\ell}(x)$ is the least non-negative residue
$qx \mod{q^{\ell}-1}$. Note that $\delta_{\ell}$ is a bijection  and is the $ \ell$-th iterate of Frobenius. 

We extend $\delta_{\ell}:\{0,1,\ldots, q^{\ell}-1\}^{n}\to \{0,1,\ldots, q^{\ell}-1\}^{n}$ by the action of
$\delta_{\ell}$ on each co-ordinate, which induces a function on $ L(D, q^{\ell}, \bb_{\ell })$. $\delta^{\ell}_{\ell}$ is the identity map and $\delta_{\ell}(\bb_{\ell})= \bb_{\ell}$

\begin{lemma} \label{3.1}  
If $ \bu  \in \{0,1,\ldots, q^{\ell}-1\}^{n}$ then
\[
 \sum_{k=0}^{\ell-1}\delta_{\ell}^{k}(\bu)= \frac{q^{\ell}-1}{q-1} ( \sigma_{q}(u_{1}),\ldots,\sigma_{q}(u_{n}) ).
 \]
\end{lemma}

\begin{proof} If $u=\sum_{j=0}^{\ell-1} a_j q^j$ with each $a_j\in [0, q-1]$ then
 $\delta_{\ell}^{k}(u)=\sum_{j=0}^{\ell-1} a_{j-k \pmod \ell} q^j$ so that 
 \[
 \sum_{k=0}^{\ell-1} \delta_{\ell}^{k}(u)=  \sum_{i=0}^{\ell-1} a_i \sum_{j=0}^{\ell-1}   q^j = 
 \frac{q^{\ell}-1}{q-1} \sum_{i=0}^{\ell-1} a_i =  \frac{q^{\ell}-1}{q-1}\sigma_{q}(u).
 \]
 The result follows applying this to each co-ordinate.
\end{proof}

 For each $\ell \geq 1$ define the function 
$\varphi_{\ell}:L(D, q^{\ell}, \bb_{\ell}) \longrightarrow \Z^{m}$  by
 \[
  \varphi_{\ell}(\bu): = \frac{\sum_{i=1}^{n}u_{i}\bd_{i}+ \bb_{\ell } }{q^{\ell}-1}
  \]
For each $ \bu \in L(D, q^{\ell}, \bb_{\ell})$, the map
$\Phi_{\bu}: \Z / \ell\Z \longrightarrow \Z^{m}$ is defined by
 \[
 \Phi_{\bu}(t):=\varphi_{\ell}(\delta^{t}_{\ell}(\bu)).
 \]

\begin{lemma} \label{Dj} 
The image of $\Phi_{\bu}$ is contained in $ \prod^{m}_{j=1}\{D_{j}^-,\ldots, D_{j}^+\}$
for all $\bu \in L(D, q^{\ell}, \bb_{\ell})$ and all $ \ell\geq 1 $ where
\[
 D_{j}^+=\sum_{i,~d_{ij}>0} d_{ij} \text{ and } D_{j}^-=\sum_{i,~d_{ij}<0} d_{ij}.
\] 
\end{lemma}

\begin{proof} If $\bu=(u_{1},\ldots, u_{n})\in L(D, q^{\ell}, \bb_{\ell})$ 
we have
\[
\sum_{i, ~d_{ij}<0}u_{i}d_{ij}\leq \sum_{i=1}^{n}u_{i }d_{ij}   =\sum_{i, ~d_{ij}>0}u_{i}d_{ij}+\sum_{i, ~d_{ij}<0}u_{i}d_{ij} \leq \sum_{i, ~d_{ij}>0}u_{i}d_{ij}
\]
and so 
\[
(q^\ell-1)D_j^-\leq \sum_{i, ~d_{ij}<0}u_{i}d_{ij}\leq \sum_{i=1}^{n}u_{i }d_{ij}  \leq \sum_{i, ~d_{ij}>0}u_{i}d_{ij} \leq (q^\ell-1)D_j^+.
\]
Moreover $ 0 \leq b_{\ell j}< q^{\ell}-1$ and so
\[
D_j^- \leq \frac{1}{q^{\ell}-1}\Bigg(\sum_{i=1}^{n}u_{i}d_{ij}+ b_{\ell j }\Bigg) < D^{+}_{j}+1
\]
but the quotient is an integer as $\bu \in L(D, q^{\ell}, \bb_{\ell})$ and so must lie in $\{D^{-}_{j},\ldots, D^{+}_{j}\}$. This implies that
$$\varphi_{\ell}(\bu)\in \prod^{m}_{j=1}\{D^{-}_{j},\ldots, D^{+}_{j} \}  \text{ for all }  \bu=(u_{1},\ldots, u_{n})\in L(D, q^{\ell}, \bb_{\ell}).$$
As  $\delta^{t}_{\ell}(\bu) \in L(D, q^{\ell}, \bb_{\ell}) $ and 
$ \Phi_{\bu}(t)=\varphi_{\ell}(\delta^{t}_{\ell}(\bu)),$ we write te follwing:
 
For $ t\in \Z / \ell\Z$, replace $\bu$ by $\delta^{t}_{\ell}(\bu)$ to obtain that $\Phi_{\bu}( t)\in \prod^{m}_{j=1}\{D^{-}_{j},\ldots, D^{+}_{j}\}$.
\end{proof}

\begin{lemma} \label{l15} For  $ \bu=(u_{1},\ldots,u_{n})\in L(D, q^{\ell}, \bb_{ \ell})$
write $u_i=\sum_{i=0}^{\ell-1} u_{i,j}q^j$ in base $q$. Then
\begin{itemize}
\item[(i)] $ qu_{i}-\delta_{\ell}(u_{i})= u_{i,\ell-1}(q^{\ell}-1)$,
\item[(ii)]
$ q \Phi_{\bu}(0)-\Phi_{\bu}(1)=\sum^{n}_{i=1}u_{i ,\ell-1}\bd_{i}+ \bb$, and
\item[(iii)] 
$ q \Phi_{\bu}(\ell-(k+1))-\Phi_{\bu}(\ell-k)=\sum^{n}_{i=1}u_{i,k \mod \ell}\bd_{i}+ \bb.$
\end{itemize}
\end{lemma}

\begin{proof} (i) \ Now
\[
qu_{i}     = u_{i,0}q+u_{i,1}q^{2}+\ldots+u_{i,\ell-2}q^{\ell-1}+u_{i,\ell-1} + u_{i, \ell-1}(q^{\ell}-1)
\]
so that 
\[
\delta_{\ell}(u_{i}) = u_{i,\ell-1}+qu_{i,0}+\ldots+ u_{i,\ell-2}q^{\ell-1}
\]
 and therefore  $qu_{i}-\delta_{\ell}(u_{i})= u_{i,\ell-1}(q^{\ell}-1)$.

\ (ii)\ By the definition of $\varphi_{\ell}$, we have 
$$ q(q^{\ell}-1)\varphi_{\ell}(\bu)= \sum^{n}_{i=1}qu_{i}\bd_{i}+q\bb_\ell \text{ and } (q^{\ell}-1)\varphi_{\ell}(\delta_{\ell}(\bu))= \sum^{n}_{i=1}\delta_{\ell}( u_{i})\bd_{i}+\delta_\ell(\bb_{\ell}),$$
so that 
\begin{align*} (q^{\ell}-1) ( q\varphi_{\ell}(\bu)-\varphi_{\ell}(\delta_{\ell}(\bu)) ) &= \sum^{n}_{i=1}( qu_{i}-\delta_{\ell}(u_{i}))\bd_{i} +(q-1)\bb_{\ell}\\
&=(q^{\ell}-1) \sum^{n}_{i=1}u_{i,\ell-1}\bd_{i}+(q-1)\bb_{\ell}
\end{align*}
since  $\delta_\ell(\bb_{\ell})=\bb_{\ell}$, and then by (i). The result follows by 
  dividing both sides by $q^\ell-1$, since $\Phi_{\bu}(0)=\varphi_{\ell}(\bu)$ and $\Phi_{\bu}(1)=\varphi_{\ell}(\delta_{\ell}(\bu))$.

(iii)\ We replace $\bu$ by $\delta^{\ell-k-1}(\bu)$ in (ii).
\end{proof}

\begin{lemma} \label{lm}
Let $ \bu \in L(D, q^{\ell}, \bb_{\ell})$ and $ 1\leq t \leq \ell-1$. Write $u_{i}=q^{t}w_{i}+v_{i}$ with $0\leq v_{i}<q^{t}$  for each   $ 1\leq i \leq n$. We have
\[
\sum_{i=1}^{n}v_{i}\bd_{i}+ \bb_{t}= q^{t}\Phi_{\bu}(-t)- \Phi_{\bu}(0) \text{ and }
\sum_{i=1}^{n}w_{i}\bd_{i}+\bb_{\ell-t}= q^{\ell-t}\Phi_{\bu }(0)-\Phi_{\bu}(-t).
\]
\end{lemma}

\begin{proof} By considering the base-$q$ digits we have 
\[
v_i= \sum_{k=0}^{t-1} u_{i,k}q^k \text{ and } w_{i}=\sum^{\ell-1}_{k=t}u_{i,k}q^{k-t},
\]
so that by lemma \ref{l15}(iii) (and writing  $\Phi_{\bu}(u-r)=\Phi_{\bu}(-r)$)
\begin{align*}
\sum_{i=1}^{n}v_{ i}\bd_{i} &=\sum_{k=0}^{t-1} q^{k}\Bigg(\sum^{n}_{i=1}u_{i,k}\bd_{ i}\Bigg) 
=\sum_{k=0}^{t-1} q^{k} (q \Phi_{\bu}(-(k+1))-\Phi_{\bu}(-k) -\bb  ) \\
&= q^t \Phi_{\bu}(-t) -\Phi_{\bu}(0) - \bb_t
\end{align*}
and
\begin{align*}
 \sum_{i=1}^{n}w_{ i}\bd_{i}& =\sum^{\ell-1}_{k=t}q^{k-t}\Bigg(\sum^{n}_{i=1}u_{i,k}\bd_{ i}\Bigg) 
  =\sum^{\ell-1}_{k=t}q^{k-t} (q \Phi_{\bu}(-(k+1))-\Phi_{\bu}(-k) -\bb  ) \\
  &= q^{\ell-t}\Phi_{\bu }(0)-\Phi_{\bu}(-t)-\bb_{\ell-t}
\end{align*}
and the result follows.
\end{proof}

\begin{proposition} \label{p30}
The set $\{\frac{\sigma(D, q^{\ell}, \bb_{\ell})}{\ell}, \ell \geq 1 \}$ admits a minimum which is attained for some $\bu\in L(D, q^{\ell}, \bb_{\ell})$ with $\ell\leq \prod^{m}_{j=1}(D_{j}^+-D_j^-+1)$ where $\Phi_{\bu}$ is injective.
\end{proposition}

\begin{proof} Let $\mathcal L$ be the set of pairs $(\bu,\ell)$ for which $ \bu\in L(D, q^{\ell}, \bb_{\ell})$ and
$\Phi_{\bu}:\Z / \ell\Z \longrightarrow \prod^{m}_{j=1}\{D_{j}^-,\ldots, D_{j}^+\} $ is injective.

For any $ \bu\in L(D, q^{\ell}, \bb_{\ell})$  with $(\bu,\ell)\not\in \mathcal L$, there exist  integers  $0\leq k_{1}< k_{2}\leq \ell-1 $ such that $\Phi_{\bu}(k_{1})= \Phi_{\bu}(k_{2})$. Replacing $ \bu $ by $\delta^{k_{1}}_{\ell}(\bu)$ and letting $ t=k_{2}-k_{1} $, we obtain an integer $t\in [1,\ell-1]$ for which
\[
\Phi_{\bu}(0)=\Phi_{\bu}(t).
\]

Now $ \bb_{\ell} =\frac{q^{\ell}-1}{q-1}\bb=\frac{q^{t}-1}{q-1}\bb q^{\ell-t}+ \frac{q^{\ell-t}-1}{q-1}\bb$ so that
 $ \bb_{\ell} = q^{\ell-t}\bb_{t}+ \bb_{\ell-t}$. Replacing $t$ by $\ell-t$ in Lemma \ref{lm} we obtain
\[
\sum_{i=1}^{n}v_{i}\bd_{i}+ \bb_{\ell-t}= q^{\ell-t}\Phi_{\bu}(t)- \Phi_{\bu}(0) =(q^{\ell-t}-1) \Phi_{\bu}(0)
\]
and similarly
\[
\sum_{i=1}^{n}w_{i}\bd_{i}+\bb_{t}= (q^{t}-1)\Phi_{\bu }(0).
\]
Therefore $\bv=(v_{1},\ldots, v_{n}) \in L(D, q ^{\ell-t}, \bb_{\ell-t})$ and  $\bw =(w_{1},\ldots, w_{n})\in L(D, q^{t}, \bb_{t})$ and $ \sigma_{p}(\bu) = \sigma_{p}(\bv)+\sigma_{p}(\bw)$ by definition, so that 
\[
\frac{ \sigma_{p}(\bu) }{\ell} =  \Big(1-\frac{t}{\ell}\Big) \frac{ \sigma_{p}(\bv) }{\ell-t} + 
\frac{t}{\ell} \frac{ \sigma_{p}(\bw) }{t} \geq  \min\Bigg\{ \frac{ \sigma_{p}(\bv) }{\ell-t} , \frac{ \sigma_{p}(\bw) }{t} \Bigg\} .
\]

Now if $(\bv,\ell-t)\not\in \mathcal L$ or $(\bw,t)\not\in \mathcal L$ then we can repeat the process, and so
\[
\frac{ \sigma_{p}(\bu) }{\ell} \geq \min_{(\bv,h) \in \mathcal L} \frac{ \sigma_{p}(\bv) }{h} =
\min_{h: \exists (\bv,h) \in \mathcal L}   \frac{ \sigma(D ,q^{h},\bb_{h}) }{h}.
\] 
 
Now if $\ell>\prod^{m}_{j=1}( D_{j}^+-D_j^-+1)$ then $ \Phi_{\bu}:\Z / \ell\Z \longrightarrow \prod^{m}_{j=1}\{D_{j}^-,\ldots, D_{j}^+\} $ cannot be injective since $\ell$ is larger than the image
and so $(u,\ell)\not\in \mathcal L$. The result follows
 \end{proof}

\begin{definition} \label{d1}
 For $\bu \in L(D, q^{\ell}, \bb_{\ell})$ define the rational number 
 \[
 s_{p}(\bu):= \frac{\sigma_{p}(\bu)}{f\ell(p-1)}
 \]
(since $v_{\pi}(x)= f(p-1) v_{q}(x)$ for all $ x \in \Q_{p}( \zeta_{q-1})$
and let
\[
s_{p}(D,\bb):=\min \Bigg \{\frac{\sigma(D, q^{\ell}, \bb_{\ell})}{f\ell(p -1)},\ell\geq 1 \Bigg \}.
\]
\end{definition}
 
\begin{proposition}  \label{prop}
We have
\[
 v_{q}( S_{\ell}(F,\bb)) \geq \ell s_{p}(D,\bb) \text{ for all } \ell\geq 1.
 \]
\end{proposition}

\begin{proof}
We deduce from the congruence \eqref{S} and the proposition \ref{p30} that for all $\ell$, 
$$v_{\pi}( S_{\ell}(F, \bb))\geq \sigma(D, q^{\ell}, \bb_{\ell}) \geq f\ell(p- 1)s_{p}(D, \bb).$$
which implies that

$$ v_{q}( S_{\ell}(F,\bb))= \frac{1}{f(p-1)}v_{\pi}( S_{\ell}(F, \bb))\geq \ell s_{p}(D, \bb). \qedhere$$ 
\end{proof}

%%%%%%%%%%%%%%%%%%%%%
\section{A bound for the valuations of the reciprocal poles and roots of the $L$-function}

We again assume that all the variables appear in the Laurent polynomial $F$. 
We define $\mu(F, \bb)$ to be the minimum of the 
$v_{q}(\alpha_{i}(F, \bb))$ and $ v_{q}(\beta_{j}(F,\bb)) $, varying over the zeros and poles of 
\[
L(F,\bb, T)=\exp \Bigg (\sum_{\ell\geq 1}S_{\ell}(F,\bb)\frac{T^{\ell}}{\ell }\Bigg)  =\frac{\prod^{s}_{i=1}(1-\beta_{i}(F, \bb)T)}{\prod^{t }_{j=1}(1-\alpha_{j}(F, \bb)T)}.
 \]
Were $ T $ is a formal variable.

\begin{theorem} \label{pL1}
We have $\mu(F, \bb) \geq s_{p}(D, \bb)$.
\end{theorem}

\begin{proof}
Now
\[
\sum_{\ell\geq 1} S_{\ell}(F, \bb)T^{\ell- 1} = \frac{d}{dT}(\ln(L(F,\bb,T)))=   -\sum^{s}_{i=1}\frac{\beta_{i}(F,\bb )}{1-\beta_{i}(F,\bb)T}+\sum^{t}_{j=1}\frac{\alpha_{j}(F,\bb)}{1-\alpha_{j}(F,\bb)T}.
\]
To study convergence of the series $\sum_{\ell\geq 1} S_{\ell}(F, \bb)T^{\ell- 1}$, $ T $ is considerd a complex(or $p$-adic) variable.
 
The right-hand side converges if and only if 
\[ 
| T |_{q}< \min_{i,j} \{| \alpha_{j}(F,\bb)|^{-1}_{q}, | \beta_{i}(F,\bb)|^{-1}_{q} \}=: q^{\mu(F, \bb) }.
\]
since then each $| \alpha_{i}(F,\bb)T |_{q} < 1$ and each $| \beta_{j}(F,\bb)T |_{q} < 1$. Therefore the left-hand side converges. Now $|S_{\ell}(F, \bb)|_q\leq q^{- \ell s_{p}(D,\bb)}$ by Proposition \ref{prop}, and so
\[
|S_{\ell}(F, \bb)T^\ell |_q \leq (|T|_qq^{- s_{p}(D,\bb)})^\ell   \leq ( q^{\mu(F, \bb)- s_{p}(D,\bb)})^\ell 
\]
which implies the result since the  left-hand side converges. 
\end{proof}

\begin{lemma} \label{L2}
There exists an extension $\F_{q^{\ell}}$ of $\F_{q}$ and a polynomial $F $ having its exponents in $D$ and its coefficients in $\F_{q^ {\ell}}$ such that  
\[
\min_{i,j}\{ v_{q^\ell}(\alpha_{j}(F,\bb)), v_{q^\ell}(\beta_{i}(F, \bb)) \} = s_{p}(D, \bb).
\]
\end{lemma}

To prove this lemma, we need the following result

\begin{theorem} \label{t28} 
Let $ H $ be the vector subspace of 
$ \F_{q}[X_{1}, X^{-1}_{1},\ldots,X_{m}, X^{-1}_{m}]$
 generated by monomials
$ \bX^{\bd_{1}},\ldots,\bX^{\bd_{n}}$ with each $\bd_{i}\in \Z^m$.
Let $ \bb\in \{0,\ldots, q-2\}^{m}$ such that $ \bb\neq 0$.
There exists a  polynomial $ F $ in $ H $ such that $ v_\pi(S_{1}(F,\bb))=\sigma(D, q,\bb)$.
\end{theorem}

\begin{proof}
In   \eqref{S} we have
\[
 S_{\ell}(F,\bb_{\ell})/\pi^{\sigma(D,q^{\ell},\bb_{\ell})}\equiv Q_{\bb}( a_{1},\ldots,a_{n}):= \sum_{\bu \in G} \prod^{n}_{i=1}\Bigg(\frac{\omega(a_{i})^{u_{i}}}{\rho_{q}(u_{i})}\Bigg)\mod\pi.
 \]
For $\ell=1$, $ Q_{\bb}(a_{1},\ldots, a_{n})$ 
is a polynomial in $a_{i}$ whose degree in each variable does not exceed $q-1$. It therefore satisfies the hypotheses of Moreno [\cite{OM04}, lemme 10], and has at least one non-zero monomial, so it does not cancel everywhere.
\end{proof}

We now establish the proof of lemma \ref{L2}.
 
\begin{proof}  
We select $\ell$ such that $\sigma(D, q^\ell, \bb_{\ell})=s_{p}(D, \bb)\ell f(p-1)$ (which exists by proposition \ref{p30}). By theorem \ref{t28}, we select a polynomial $F$ for which $v_{\pi}(S_{\ell}(F,\bb_{\ell}))=\sigma(D , q^{\ell},\bb_{\ell})=s_{p}(D, \bb)\ell f(p-1)$. Since 
\[
S_{\ell}(F,\bb_{\ell})=\sum_i\beta_i(F, \bb)-\sum_j\alpha_j(F,\bb),
\]
we have
\begin{align*}
s_{p}(D, \bb)\ell f(p-1)=v_{\pi}(S_{\ell}(F,\bb_{\ell}))& \geq \min_{ i,j}\{v_{\pi}(\alpha_{j}(F,\bb)), v_{\pi}(\beta_{i}(F, \bb))\}\\
           &\geq s_{p}(D, \bb)\ell f(p-1)
\end{align*}
by proposition \ref{pL1}, and hence we obtain ence the equality sought. 
\end{proof}
  
  \section{Lower bounds in practice}

Let  $D\subset  \ZZ$. For any   $\bu\in L(D, q^{\ell}, \bb_{\ell})$ we have  
\[
f\ell(p-1)\leq \sigma_p\left(\sum^{n}_{i=1} u_{i}\bd_{i}+\bb_{\ell}\right) \leq \sum^{n}_{i=1} \sigma_p(u_{i})\sigma_p(\bd_{i})+\sigma_p(\bb_{\ell})
\]
by (\ref{L}) and Moreno [\cite{OM04}, Proposition 11]. 
Hence, we deduce that
\[
 \sigma_{p}(\bu)=\sum^{n}_{i=1} \sigma_p(u_{i})\geq \frac{f\ell(p-1)-\sigma_p(\bb_{\ell})}{\max_i\sigma_p(\bd_{i})}=\ell\frac{f(p-1)-\sigma_p(\bb)}{\max_i\sigma_p(\bd_{i})}
\]
and so $\sigma(D, q^{\ell}, \bb_{\ell})\geq \ell\frac{f(p-1)-\sigma_p(\bb)}{\max_i\sigma_p(\bd_{i})}$ and
\[
s_p(D,\bb)\geq \frac{f(p-1)-\sigma_p(\bb)}{f(p-1)\max_i\sigma_p(\bd_{i})}
\]
(and $ v_{q}(S_{\ell}(D,b))\geq \ell  \frac{f(p-1)-\sigma_p(\bb)}{f(p-1)\max_i\sigma_p(\bd_{i})}$ by 
proposition \ref{prop}).

\begin{example} Let  $q=p, D=\{1, \ldots, d \}$ and  $b\in \{0,\ldots,p-2\}$  so that $\sigma_p(\bb)=b$ and
$\bd_{i}=i.$ And so the above gives the lower bound $ \sigma_{p}(\bu)\geq \ell \frac{p-1-b}{d}$ as $\max_i\sigma_p(\bd_{i})=d$. 
If $d$ divides $p-1-b$, say  $p-1-b=ds$ then this yields $ \sigma_{p}(\bu)\geq s$ when $ \ell=1$.
But we attain this bound from   the solution $u_1=\cdots=u_{d-1}=0, u_{d}=s$.
Note that these solutions are obtained by solving the system of congruences  (\ref{eq; congruence}) for $ m=1. $ 
\end{example}

\begin{example} Let $D=\{1,\ldots,p^a-1\}$, $b=p^c-1$, $f=ka+c$, for given $a,c,k\in\Z_{\geq 1}$ so that
  $\max_i\sigma_p(\bd_{i})=a(p-1)$, $\sigma_p(b)=c(p-1)$. When $\ell=1$  the above lower bound becomes
 $\sigma_{p}(\bu)\geq  { \frac{f(p-1)-\sigma_p(\bb)}{\max_i\sigma_p(\bd_{i})}} = \frac{f-c}{a}=k$. But we attain this bound from   the solution $u_1=\ldots=u_{p^a-2}=0, u_{p^a-1}=p^c\frac{p^{ka}-1}{p^a-1}=p^{c+a}+p^{c+2a}+\dots+p^{c+(k-1)a}$, since
  $\sigma_{p}(\bu)=\sigma_{p}(u_{p^a-1})=k$ wich are obtained by solving  (\ref{eq; congruence}) for $ m=1 .$
\end{example}

\bibliographystyle{amsplain}

\end{document}